\documentclass[11pt,reqno,a4paper]{amsart}

\usepackage{amsfonts, amsmath, amsthm, amssymb}
\usepackage[english]{babel}
\usepackage[mathcal]{eucal}

\allowdisplaybreaks

\title[]{An application of approach theory to the relative Hausdorff measure of non-compactness for the Wasserstein metric}

\author{Ben Berckmoes, Tim Hellemans, Mark Sioen and Jan Van Casteren}
\keywords{approach theory, Ascoli's Theorem, contractive approach structure, Dini's Theorem, (relative) Hausdorff measure of non-compactness, measure of non-uniform integrability, Prokhorov's Theorem, Wasserstein metric}
\thanks{The first author is post doctoral fellow at the Fund for Scientific Research of Flanders (FWO)}

\date{}

\DeclareMathOperator*{\mylim}{li\vphantom{p}m}
\DeclareMathOperator*{\myinf}{in\vphantom{p}f}

\begin{document}

\maketitle

\newtheorem{pro}{Proposition}[section]
\newtheorem{lem}[pro]{Lemma}
\newtheorem{thm}[pro]{Theorem}
\newtheorem{de}[pro]{Definition}
\newtheorem{co}[pro]{Comment}
\newtheorem{no}[pro]{Notation}
\newtheorem{vb}[pro]{Example}
\newtheorem{vbn}[pro]{Examples}
\newtheorem{gev}[pro]{Corollary}
\newtheorem{vrg}[pro]{Question}
\newtheorem{rem}[pro]{Remark}

\begin{abstract}
After shortly reviewing the fundamentals of approach theory as introduced by R. Lowen in 1989, we show that this theory is intimately related with the well-known Wasserstein metric on the space of probability measures with a finite first moment on a complete and separable metric space. More precisely, we introduce a canonical approach structure, called the contractive approach structure, and prove that it is metrized by the Wasserstein metric. The key ingredients of the proof of this result are Dini's Theorem, Ascoli's Theorem, and the fact that the class of real-valued contractions on a metric space has some nice stability properties. We then combine the obtained result with Prokhorov's Theorem to establish inequalities between the relative Hausdorff measure of non-compactness for the Wasserstein metric and a canonical measure of non-uniform integrability.
\end{abstract}

\section{Introduction}

Approach spaces were introduced in 1989 by R. Lowen as a unification of metric spaces and topological spaces (\cite{L89},\cite{Lo},\cite{L15},\cite{SV16}). Instead of quantifying the properties of a space $X$ by means of one metric, an approach space is determined by assigning to each point $x \in X$ a collection $\mathcal{A}_x$ of $[0,\infty]$-valued maps on $X$ which are interpreted as local distances based at $x$. By imposing suitable axioms on the collections $\mathcal{A}_x$, we obtain a structure on $X$ which, as in the case of metric spaces, allows us to deal with quantitative concepts such as asymptotic radius and center (\cite{AMS82},\cite{B85},\cite{L81},\cite{L01}) and Hausdorff measure of non-compactness (\cite{BG80},\cite{L88}) as used in functional analysis, especially in various areas of approximation theory (\cite{AMS82}), fixed-point theory (\cite{GK90}), operator theory (\cite{AKPRS92},\cite{PS88}), and Banach space geometry (\cite{DB86},\cite{KV07},\cite{WW96}). However, unlike metric spaces, approach spaces share many of the `structurally good' properties of topological spaces, such as e.g. the easy and canonical formation of product spaces.

The structural flexibility of approach theory entails the existence of canonical approach spaces in branches of mathematical analysis such as functional analysis (\cite{LS00},\cite{SV03},\cite{LV04},\cite{SV04},\cite{SV06},\cite{SV07}), hyperspace theory (\cite{LS96},\cite{LS98},\cite{LS00'}), domain theory (\cite{CDL11},\cite{CDL14},\cite{CDS14}), and probability theory and statistics (\cite{BLV11},\cite{BLV13},\cite{BLV16}). A careful study of these approach spaces has resulted in new insights and applications in these branches.

In this paper, we will use approach theory to study the relative Hausdorff measure of non-compactness for the well-known Wasserstein metric (\cite{V03}) on the space of probability measures with a finite first moment on a separable and complete metric space. The paper is structured as follows.

A brief overview of the basic notions of approach theory is given in Section \ref{sec:AppThy}. For a deep and complete treatment of the topic, we refer the reader to (\cite{Lo}) and (\cite{L15}).

In Section \ref{sec:WMAP}, we show that the Wasserstein metric is intimately related to approach theory. We introduce a canonical approach structure, called the contractive approach structure, on the set of probability measures with a finite first moment on a separable and complete metric space, and we prove that it is metrizable by the Wasserstein metric. The main ingredients of the proof consist of Dini's theorem, Ascoli's Theorem, and the fact that the class of contractions of a metric space into the real numbers has nice stability properties.

The relative Hausdorff measure of non-compactness for the Wasserstein metric is investigated in Section \ref{sec:HMWM}. Approach theory, and more precisely Theorem \ref{thm:AkW} obtained in Section \ref{sec:WMAP}, and Prokhorov's Theorem will turn out to be the essential tools to obtain in Theorem \ref{thm:muUImuH} inequalities between the Hausdorff measure of noncompactness for the Wasserstein metric and a canonical measure of non-uniform integrability. 

\section*{Acknowledgements}

The authors thank Bob Lowen and Yvik Swan for stimulating and interesting discussions.

\section{Approach theory}\label{sec:AppThy}

\subsection{Approach spaces}

Let $X$ be a non-empty set and $[0,\infty]^X$ the collection of maps of $X$ into $[0,\infty]$. For maps $\phi_1,\phi_2 \in [0,\infty]^X$ we shall always interpret their maximum and minimum pointwise, e.g. $\max\{\phi_1,\phi_2\}(x) = \max\{\phi_1(x),\phi_2(x)\}$. A subcollection $\mathcal{A}_0 \subset [0,\infty]^X$ is said to be {\em upwards-directed} iff for all $\phi_1, \phi_2 \in \mathcal{A}_0$ there exists $\phi \in \mathcal{A}_0$ such that $\max\{\phi_1,\phi_2\} \leq \phi$. A {\em functional ideal} on $X$ is an upwards-directed collection $\mathcal{A}_0 \subset [0,\infty]^X$ such that for each $\phi \in [0,\infty]^X$  
$$\left(\forall \varepsilon > 0, \forall \omega < \infty, \exists \phi_0 \in \mathcal{A}_0: \min\{\phi,\omega\} \leq \phi_0 + \varepsilon\right) \Rightarrow \phi \in \mathcal{A}_0.$$
Notice that functional ideals are closed under the formation of finite maxima.

An {\em approach structure} on $X$ is an assignment $\mathcal{A}$ of a functional ideal $\mathcal{A}_x$ on $X$ to each point $x \in X$ such that for each $x \in X$ and each $\phi \in \mathcal{A}_x$
\begin{itemize}
	\item[(A1)] $\phi(x) = 0,$
	\item[(A2)] $\forall \varepsilon > 0, \forall \omega < \infty, \exists \left(\phi_x\right)_x \in \Pi_{x \in X} \mathcal{A}_x, \forall y,z \in X :$
	\begin{displaymath}
	\min\{\phi(y),\omega\} \leq \phi_x(z) + \phi_z(y) + \varepsilon.
	  \end{displaymath}
\end{itemize}
If $\mathcal{A}$ is an approach structure on $X$, then $(X,\mathcal{A})$ is called an {\em approach space}. In an approach space $(X,\mathcal{A})$, a map $\phi \in \mathcal{A}_x$ is interpreted as a local distance based at $x$.

A {\em basis for an approach structure $\mathcal{A}$} on $X$ is an assignment $\mathcal{B}$ of a collection $\mathcal{B}_{x} \subset \mathcal{A}_x$ to each point  $x \in X$ such that for each $x \in X$
\begin{displaymath}
\forall \phi \in \mathcal{A}_x, \forall \varepsilon > 0, \forall \omega < \infty, \exists \psi \in \mathcal{B}_{x} : \min\{\phi,\omega\} \leq \psi + \varepsilon.
 \end{displaymath}
 We will also say that $\mathcal{B}$ {\em generates} $\mathcal{A}$.
 
The following result provides a common method to introduce approach structures on a set.

\begin{pro}\label{pro:ApproachGems}
Let $\mathcal{B}$ be an assignment of a non-empty collection $\mathcal{B}_x \subset [0,\infty]^X$ to each point $x \in X$. Then there exists a unique approach structure $\mathcal{A}$ on $X$ such that $\mathcal{B}$ is a basis for $\mathcal{A}$ if and only if each $\mathcal{B}_{x}$ is upwards-directed and for each $x \in X$ and $\psi \in \mathcal{B}_{x}$
\begin{itemize}
	\item[(BA1)] $\psi(x) = 0,$
	\item[(BA2)] $\forall \varepsilon > 0, \forall \omega < \infty, \exists \left(\psi_x\right)_x \in \Pi_{x \in X} \mathcal{B}_{x}, \forall y,z \in X : $ 
	\begin{displaymath}
	\min\{\psi(y),\omega\} \leq \psi_x(z) + \psi_z(y) + \varepsilon.
	\end{displaymath}
\end{itemize}
\end{pro}

Recall that a {\em metric} on $X$ is a map $m$ of $X \times X$ into $\mathbb{R}^+$ such that
\begin{itemize}
	\item[(M1)] $\forall x \in X : m(x,x) = 0,$
	\item[(M2)] $\forall x,y \in X : m(x,y) = m(y,x),$
	\item[(M3)] $\forall x,y,z \in X : m(x,y) \leq m(x,z) + m(z,y).$
\end{itemize}

Let $m$ be a metric on $X$ and assign to each point $x \in X$ the collection
\begin{displaymath}
\mathcal{B}_{m,x} = \{m(x,\cdot)\},
\end{displaymath}
where $m(x,\cdot)$ stands for the map
\begin{displaymath}
m(x,\cdot) : X \rightarrow \mathbb{R}^+ : y \mapsto m(x,y).
\end{displaymath}
Then it follows from Proposition \ref{pro:ApproachGems} that there exists a unique approach structure $\mathcal{A}_m$ on $X$ such that $\left(\mathcal{B}_{m,x}\right)_{x \in X}$ is a basis for $\mathcal{A}_m$. It is not hard to establish that
\begin{displaymath}
\mathcal{A}_{m,x} = \left\{\phi \in \left(\mathbb{R}^+\right)^X \mid \phi \leq m(x,\cdot)\right\}.
\end{displaymath}
We call $\mathcal{A}_m$ the {\em approach structure underlying $m$}\index{underlying!approach structure of a metric}.  An approach structure $\mathcal{A}$ is said to be {\em metrizable}\index{metrizable approach structure} iff there exists a metric $m$ such that $\mathcal{A} = \mathcal{A}_{m}$.

Approach spaces form the appropriate axiomatic setting for a quantitative analysis based on numerical indices measuring to what extent topological properties (such as convergence and compactness) fail to be valid. In the next subsections, we give a brief outline of how such an analysis can be developed. Throughout,  $X = \left(X,\mathcal{A}\right)$ will be an approach space. 

\subsection{The associated topology}

For $x \in X$, $\phi \in \mathcal{A}_x$, and $\varepsilon > 0$, we define the {\em $\phi$-ball with center $x$ and radius $\varepsilon$} as the set 
$$B_{\phi}(x,\varepsilon) = \left\{y \in X \mid \phi(y) < \varepsilon\right\}.$$
More loosely, we will also refer to the latter set as a ball with center $x$ or a ball with radius $\varepsilon$. 

Consider a point $x \in X$ and a set $A \subset X$. We say that $x$ belongs to the {\em closure of $A$}\index{closure} iff each ball with center $x$ contains a point of $A$, and to the {\em interior of $A$}\index{interior} iff $A$ contains a ball with center $x$. We denote the closure of $A$ as $\overline{A}$, and the interior of $A$ as $A^\circ$.
We call a set $A \subset X$ {\em closed}\index{closed set} iff $\overline{A} = A$, and {\em open}\index{open set} iff $A^\circ = A$. 

It is not hard to establish that the collection of closed sets in $X$ contains $\emptyset$ and $X$ and is closed under the formation of finite unions and arbitrary intersections. Furthermore, a set is open if and only if its complement is closed. In particular, the collection of open sets in $X$ contains $\emptyset$ and $X$ and is closed under the formation of finite intersections and arbitrary unions. We infer that the open sets in $X$ define a topology, $\mathcal{T}_{\mathcal{A}}$, which will be called the {\em topology associated with} $\mathcal{A}$. If $\mathcal{A} = \mathcal{A}_m$ for a metric $m$, then the topology associated with $\mathcal{A}$ coincides with the usual topology derived from $m$.  

In the remainder of this section, each topological notion (such as e.g. convergence and compactness) should be interpreted in the topological space $(X,\mathcal{T}_{\mathcal{A}})$.

\subsection{The associated quasimetric}

A {\em quasimetric} on $X$ is a map $d$ of  $X \times X$ into $\mathbb{R}^+$  which satisfies the properties (M1) and (M3) of a metric.

We call $X$ {\em locally bounded} iff for all $(x,y) \in X \times X$ there exists a constant $C > 0$ such that  $\phi(y) \leq C$ for all $\phi \in \mathcal{A}_x$.

If $X$ is locally bounded, put, for $x,y \in X$,
$$d_{\mathcal{A}}(x,y) = \sup_{\phi \in \mathcal{A}_x}\phi(y).$$
One easily verifies that $d_{\mathcal{A}}$ is the smallest quasimetric on $X$ with the property that $\phi \leq d_{\mathcal{A}}(x,\cdot)$ for all $x \in X$ and $\phi \in \mathcal{A}_x$. We call $d_{\mathcal{A}}$ the {\em quasimetric associated with} $\mathcal{A}$. If $\mathcal{A} = \mathcal{A}_m$ for a metric $m$, then $d_{\mathcal{A}} = m$.

Where needed, we shall assume that $X$ is locally bounded.

\subsection{Convergence}

Throughout the paper, we will make use of the convergence theory of nets. For the basic facts about this theory, we refer the reader to \cite{W70}. Consider a net $\left(x_\eta\right)_\eta$ in $X$, a point $x \in X$, and $\varepsilon > 0$. We say that $\left(x_\eta\right)_\eta$ is {\em $\varepsilon$-convergent to $x$} iff each ball $B$ with center $x$ and radius $\varepsilon$ eventually contains $(x_\eta)_\eta$. We define the {\em limit operator of $\left(x_\eta\right)_\eta$ at $x$} as 
$$\lambda\left(x_\eta \rightarrow x\right) = \inf\left\{\alpha > 0 \mid (x_\eta)_\eta \text{ is  $\alpha$-convergent to } x\right\}.$$
Notice that the latter number is zero if and only if $x_\eta \rightarrow x$. If $\mathcal{B}$ is a basis for $\mathcal{A}$, then one easily verifies that
$$\lambda(x_\eta \rightarrow x) = \sup_{\phi \in \mathcal{B}_{x}} \limsup_{\eta} \phi(x_\eta).$$
Therefore, if $\mathcal{A} = \mathcal{A}_m$ for a metric $m$,
\begin{equation}
\lambda(x_\eta \rightarrow x) = \limsup_{\eta} m(x,x_\eta).\label{eq:LimOpM}
\end{equation}
Notice that if $(x_\eta)_\eta$ is convergent (in the topological space $(X,\mathcal{T}_\mathcal{A})$), then $\myinf_{x \in X} \lambda(x_\eta \rightarrow x) = 0$. We call $X$ {\em complete} iff the converse holds, i.e. if each net $(x_\eta)_\eta$ with the property $\myinf_{x \in X} \lambda(x_\eta \rightarrow x) = 0$ is convergent. If $\mathcal{A} = \mathcal{A}_m$ for a metric $m$, then $\mathcal{A}$ is complete if and only if  $m$ is complete in the classical sense.

\subsection{Compactness}\label{subsec:compactness}

If $\left(\Phi = \left(\phi_x\right)_x\right) \in \Pi_{x \in X} \mathcal{A}_x,$
then a set $B \subset X$ is called a {\em $\Phi$-ball} iff there exist $x \in X$ and $\alpha > 0$ such that $B = B_{\phi_x}(x,\alpha)$.  Let $A \subset X$. We say that a collection $\mathcal{V}$ of sets $V \subset X$ {\em covers} $A$ iff $A \subset \cup_{V \in \mathcal{V}} V$. We call $A$ {\em$\varepsilon$-relatively compact} iff it holds for each $\Phi \in \Pi_{x \in X} \mathcal{A}_x$ that $A$ can be covered with finitely many $\Phi$-balls with radius $\varepsilon$, and we define the {\em relative measure of non-compactness of $A$} as 
$$\mu_{rc}(A) = \inf \{\alpha > 0 \mid \textrm{$A$ is $\alpha$-relatively compact}\}.$$
If $\mathcal{B}$ is a basis for $\mathcal{A}$, then
$$\mu_{rc}(A) = \sup_{\left(\phi_x\right)_x \in \Pi_{x \in X} \mathcal{B}_{x}} \myinf_{\substack{Y \subset X\\\textrm{finite}}} \sup_{a \in A} \myinf_{y \in Y} \phi_y(a).$$
Therefore, if $\mathcal{A} = \mathcal{A}_m$ for a metric $m$,
$$\mu_{rc}(A) = \myinf_{\substack{Y \subset X\\\textrm{finite}}} \sup_{a \in A} \myinf_{y \in Y} m(y,a),$$
which is known as the {\em relative Hausdorff measure of non-compactness for $m$} (\cite{BG80},\cite{L88},\cite{WW96}).

The following result links the relative measure of non-compactness to the limit operator.

\begin{thm}\label{thm:compactness}
For $A \subset X$,
\begin{displaymath}
\mu_{rc}(A) = \sup_{(x_\eta)_\eta} \myinf_{(x_{h(\eta^\prime)})_{\eta^\prime}} \myinf_{x \in X} \lambda(x_{h(\eta^\prime)} \rightarrow x),
\end{displaymath}
the supremum taken over all nets $(x_\eta)_\eta$ in $A$, and the first infimum over all subnets $(x_{h(\eta^\prime)})_{\eta^\prime}$ of $(x_\eta)_\eta$. Furthermore, if $A$ is relatively compact, then $\mu_{rc}(A) = 0$, and the converse holds if $X$ is complete.
\end{thm}

\subsection{Contractions}

Let $f$ be a map of $X$ into an approach space $Y = \left(Y, \mathcal{A}^\prime\right)$ and $x \in X$. We say that $f$ is {\em contractive at $x$} iff
$$\forall \phi \in \mathcal{A}_{f(x)}^\prime :  \phi \circ f \in \mathcal{A}_x,$$
and that $f$ is a {\em contraction}  iff it is contractive at every $x \in X$. 

The following result characterizes the contractive property in terms of the limit operator. 

\begin{thm}\label{pro:CharacContraction}
A map $f : X \rightarrow Y$ is contractive at $x$ if and only if for every net  $\left(x_\eta\right)_\eta$ in $X$
$$\lambda\left(f(x_\eta) \rightarrow f(x)\right) \leq \lambda\left(x_\eta \rightarrow x\right).$$
\end{thm}
 
One easily derives from Theorem \ref{pro:CharacContraction} that contractions between approach spaces are continuous, and that, if $\mathcal{A} = \mathcal{A}_m$ for a metric $m$ and $\mathcal{A}^\prime = \mathcal{A}_{m^\prime}$ for a metric $m^\prime$, a map $f : X \rightarrow Y$ is contractive if and only if it is contractive in the metric sense, i.e. $m^\prime(f(x),f(y)) \leq m(x,y)$ for all $x,y \in X$.

\subsection{Weak approach structures}

Let $\mathcal{A}^\prime$ be an approach structure on $X$. We say that {\em $\mathcal{A}^\prime$ is  weaker than $\mathcal{A}$} or, equivalently, that {\em $\mathcal{A}$ is stronger than $\mathcal{A}^\prime$} iff the identity map
\begin{displaymath}
1_X : \left(X,\mathcal{A}\right) \rightarrow \left(X,\mathcal{A}^\prime\right) : x \mapsto x
\end{displaymath}
is a contraction. Notice that $\mathcal{A}^\prime$ is weaker than $\mathcal{A}$ if and only if $\mathcal{A}_x^\prime \subset \mathcal{A}_x$ for each $x \in X$.

It easily follows from Theorem \ref{pro:CharacContraction} that

\begin{thm}\label{thm:characweak}
$\mathcal{A}^\prime$ is weaker than $\mathcal{A}$ if and only if, for each net $\left(x_\eta\right)_\eta$ in $X$ and each point $x \in X$,
$$\lambda_{\mathcal{A}^\prime}(x_\eta \rightarrow x) \leq \lambda_{\mathcal{A}}(x_\eta \rightarrow x).$$
In particular, $\mathcal{A}^\prime = \mathcal{A}$ if and only if, for each net $\left(x_\eta\right)_\eta$ in $X$ and each point $x \in X$,
$$\lambda_{\mathcal{A}^\prime}(x_\eta \rightarrow x) = \lambda_{\mathcal{A}}(x_\eta \rightarrow x).$$

\end{thm}

Many interesting approach structures arise naturally in various areas in mathematical analysis. This is mainly due to the fact that approach spaces were designed in such a way that they allow for a lot of `structural flexibility'. This metaprinciple is captured by the following theorem (\cite{Lo}).

\begin{thm}\label{thm:StructureThmApp}
Consider a set $Y$ and an indexed collection of maps
$$\left(f_i : Y \rightarrow Y_i\right)_{i \in I},$$
where each $Y_i = \left(Y_i,\mathcal{A}_i\right)$ is an approach space. Then there exists a weakest approach structure $\mathcal{A}_{w}$ on $Y$ with the property that each map
$$f_i: \left(Y,\mathcal{A}_{w}\right) \rightarrow Y_i$$
is contractive. Moreover, putting for $y \in Y$
$$\mathcal{B}_{y} = \left\{\max_{k \in K} \phi_k \circ f_k \mid K \subset I \textrm{ finite}, \forall k \in K : \phi_k \in \mathcal{A}_{k,f_k(y)}\right\},$$
it follows that $\left(\mathcal{B}_y\right)_{y \in Y}$ is a basis for $\mathcal{A}_w$. Finally, $\mathcal{A}_w$ is characterized by the property that it holds for every map
$$f : Z \rightarrow \left(Y,\mathcal{A}_w\right),$$
where $Z$ is an approach space, that $f$ is contractive if and only if each map
$$f_i \circ f : Z \rightarrow Y_i$$
is contractive.
\end{thm}

We call the approach structure $\mathcal{A}_w$ in Theorem \ref{thm:StructureThmApp} the {\em weak approach structure for the collection of maps}
$$\left(f_i : Y \rightarrow Y_i\right)_{i \in I}.$$

Let $\lambda_i$ stand for the limit operator in the space $Y_i$ and $\lambda_w$ for the limit operator in the space $\left(Y,\mathcal{A}_w\right)$. Then

\begin{gev}\label{gev:InAppLimOp}
For a net $\left(y_\eta\right)_\eta$ in $Y$ and a point $y \in Y$,
$$\lambda_w(y_\eta \rightarrow y) = \sup_{i \in I} \lambda_i(f_i\left(y_\eta\right) \rightarrow f_i\left(y\right)).$$
In particular,
$$y_\eta \rightarrow y \textrm{ in $\left(Y,\mathcal{A}_w\right)$ }\Leftrightarrow \forall i \in I : f_i\left(y_\eta\right) \rightarrow f_i(y) \textrm{ in $Y_i$}.$$
\end{gev}

It is important to notice that the weak approach structure for a collection of maps of a set into metrizable approach spaces in general fails to be metrizable. 

\section{The Wasserstein metric and approach theory}\label{sec:WMAP}

Let $S = (S,d)$ be a separable and complete metric space, $\mathcal{C}_b(S,\mathbb{R})$ the collection of all maps $f$ of $S$ into $\mathbb{R}$ which are bounded and continuous, and $\mathcal{P}(S)$ the set of all Borel probability measures on $S$. 

Recall that a net $(P_\eta)_\eta$ in $\mathcal{P}(S)$ is said to {\em converge weakly} to $P \in \mathcal{P}(S)$ iff  
$$\forall f \in \mathcal{C}_b(S,\mathbb{R}) : \int_S f dP_\eta \rightarrow\int_S f dP.$$
We write $P_\eta \stackrel{w}{\rightarrow} P$ to indicate that $(P_\eta)_\eta$ converges weakly to $P$. It is well known that weak convergence corresponds to convergence with respect to the {\em weak topology} on $S$ (\cite{B99},\cite{P67}). The latter is separable and completely metrizable.

Furthermore, let $\kappa(S,\mathbb{R})$ be the collection of all maps $f$ of $S$ into $\mathbb{R}$ which are contractive, i.e. for which the inequality
$$\left|f(x) - f(y)\right| \leq d(x,y)$$ 
holds for all $x, y \in S$, and let $\mathcal{P}^1(S)$ be the set of all $P \in \mathcal{P}(S)$ with a finite first moment, i.e. for which
$$\int_S d(a,\cdot) dP < \infty$$ 
for a certain (and thus, by the triangle inequality, for all) $a \in S$, where $d(a,\cdot)$ stands for the map 
$$d(a,\cdot) : S \rightarrow \mathbb{R}^+ : x \mapsto d(a,x).$$
Note that for $f \in \kappa(S,\mathbb{R})$, $P \in \mathcal{P}^1(S)$, and $a \in S$, 
$$\int_S \left|f\right| dP \leq \left|f(a)\right| + \int_S d(a,\cdot) dx < \infty,$$
from which we infer that $\int_S f dP$ is well-defined.

The {\em Wasserstein metric} on $\mathcal{P}^1(S)$, see e.g. \cite{V03}, is defined by the formula
\begin{equation*}
W(P,Q) = \myinf_{\pi} \int_{S \times S} d(x,y) d\pi(x,y),
\end{equation*} 
where the infimum runs over all Borel probability measures $\pi$ on $S \times S$ with first marginal $P$ and second marginal $Q$. By Kantorovich duality theory (\cite{V03},\cite{E10},\cite{E11}), the alternative formula
\begin{equation}
W(P,Q) = \sup_{f \in \kappa(S,\mathbb{R})} \left|\int_S f dP - \int_S f dQ\right|\label{eq:defWass}
\end{equation}
holds. Furthermore, if $S = \mathbb{R}$, then
\begin{equation*}
W(P,Q) = \int_{-\infty}^\infty \left|F_P(x) - F_Q(x)\right| dx
\end{equation*}
with $F_P$ (respectively $F_Q$) the cumulative distribution function of $P$ (respectively $Q$).

The topology underlying the Wasserstein metric is stronger than the weak topology. More precisely, for a net $(P_\eta)_\eta$ in $\mathcal{P}^1(S)$ and $P \in \mathcal{P}^1(S)$, the following are equivalent (\cite{V03}):
\begin{enumerate}
\item $W(P,P_\eta) \rightarrow 0,$
\item $P_\eta \stackrel{w}{\rightarrow} P$ and $\forall a \in S : \int_S d(a,\cdot) dP_\eta \rightarrow \int_S d(a,\cdot) dP$.
\end{enumerate}
Also, the Wasserstein metric is separable and complete, see e.g. \cite{B08}.

\begin{pro}
Let $\mathcal{A}_W$ be the underlying approach structure of the Wasserstein metric $W$. Then, in the approach space $(\mathcal{P}^1(S),\mathcal{A}_W)$, the limit operator of a net $(P_\eta)_\eta$ at a point $P$ is given by
$$\lambda_W(P_\eta \rightarrow P) = \limsup_{\eta} \sup_{f \in \kappa(S,\mathbb{R})}  \left|\int f dP - \int f dP_\eta\right|.$$
\end{pro}

\begin{proof}
This follows from (\ref{eq:LimOpM}) and (\ref{eq:defWass}) .
\end{proof}

Inspired by formula (\ref{eq:defWass}), we introduce the {\em contractive approach structure} $\mathcal{A}_\kappa$ on $\mathcal{P}^1(S)$ as the weak approach structure for the collection of maps
$$\left(\mathcal{P}^1(S) \rightarrow \mathbb{R} : P \mapsto \int f dP\right)_{f \in \kappa(S,\mathbb{R})},$$
where, of course, the approach structure underlying the Euclidean metric is considered on $\mathbb{R}$. 

\begin{pro}\label{pro:limopkappa}
In the approach space $(\mathcal{P}^1(S),\mathcal{A}_\kappa)$, the limit operator of a net $(P_\eta)_\eta$ at a point $P$ is given by
$$\lambda_\kappa(P_\eta \rightarrow P) = \sup_{f \in \kappa(S,\mathbb{R})} \limsup_{\eta} \left|\int f dP - \int f dP_\eta\right|.$$
\end{pro}

\begin{proof}
This follows from the definition of $\mathcal{A}_\kappa$ and Corollary \ref{gev:InAppLimOp}.
\end{proof}

The main goal of this section is to establish the following result.

\begin{thm}\label{thm:AkW}
The Wasserstein metric metrizes the contractive approach structure, which, by Theorem \ref{thm:characweak}, is equivalent with the assertion that, for each net $(P_\eta)_\eta$ in $\mathcal{P}^1(S)$ and each $P \in \mathcal{P}^1(S)$,
$$\lambda_\kappa(P_\eta \rightarrow P) = \lambda_{W}(P_\eta \rightarrow P),$$
or, more explicitly,
\begin{equation*}
\sup_{f \in \kappa(S,\mathbb{R})} \limsup_{\eta} \left|\int f dP - \int f dP_\eta\right| = \limsup_{\eta} \sup_{f \in \kappa(S,\mathbb{R})} \left|\int f dP - \int f dP_\eta\right|.
\end{equation*}
\end{thm}

For the proof of Theorem \ref{thm:AkW}, the authors were inspired by \cite{E10} and \cite{E11}, in particular Section 3 in \cite{E10}. We first give four lemmas.

Fix $a \in S$ and let $\kappa_a(S,\mathbb{R})$ be the set of those $f \in \kappa(S,\mathbb{R})$ for which $f(a) = 0$.

\begin{lem}\label{lem:AbsBecomesNormal}
For $P,Q \in \mathcal{P}^1(S)$,
\begin{equation}
W(P,Q) = \sup_{f \in \kappa_a(S,\mathbb{R})} \left(\int f dP - \int f dQ\right),\label{eq:NewFormW}
\end{equation}
and, for a net $(P_\eta)_\eta$ in $\mathcal{P}^1(S)$ and $P \in \mathcal{P}^1(S)$,
\begin{equation}
\lambda_\kappa(P_\eta \rightarrow P) = \sup_{f \in \kappa_a(S,\mathbb{R})} \limsup_{\eta} \left(\int f dP_\eta - \int f dP\right).\label{eq:NewFormLambdaKappa}
\end{equation}
\end{lem}

\begin{proof}
For $g \in \kappa(S,\mathbb{R})$, put 
$$g_a = g - g(a).$$ 
Then $g_a \in \kappa_a(S,\mathbb{R})$ and 
$$\int g dP - \int g dQ = \int g_a dP - \int g_a dQ.$$
Thus, from (\ref{eq:defWass}) we now learn that
$$W(P,Q) = \sup_{f \in \kappa_a(S,\mathbb{R})} \left|\int f dP - \int f dQ\right|.$$
But then (\ref{eq:NewFormW}) follows from the fact that, for any $f : S \rightarrow \mathbb{R}$, $f \in \kappa_a(S,\mathbb{R})$ if and only if $-f \in \kappa_a(S,\mathbb{R})$. By the same observations, (\ref{eq:NewFormLambdaKappa}) is established.
\end{proof}

\begin{lem}\label{lem:seqsup}
For each $\varepsilon > 0$, there exists a net  $(f_\eta)_{\eta \in D}$ in $\kappa_a(S,\mathbb{R})$, a directed set $D'$, and  a monotonically increasing and cofinal map $h :  D' \rightarrow D$ such that

\begin{equation*}
\lambda_{W}(P_\eta \rightarrow P) -  \varepsilon  \leq \lim_{\eta^\prime} \left(\int f_{h(\eta')} dP - \int f_{h(\eta')} dP_{h(\eta')}\right) \leq \lambda_{W}(P_\eta \rightarrow P).
\end{equation*}

\end{lem}

\begin{proof}
By Lemma \ref{lem:AbsBecomesNormal}, choose, for each $\eta \in D$, $f_\eta \in \kappa_a(S,\mathbb{R})$ such that 
$$W(P,P_\eta) - \varepsilon \leq  \int f_\eta dP - \int f_\eta dP_\eta \leq W(P,P_\eta) .$$
Then 
$$\lambda_{W}(P_\eta \rightarrow P) - \varepsilon \leq  \limsup_{\eta} \left(\int f_\eta dP - \int f_\eta dP_\eta \right) \leq \lambda_{W}(P_\eta \rightarrow P).$$
Since $\limsup_{\eta} \left(\int f_\eta dP - \int f_\eta dP_\eta \right)$ is the largest accumulation point of the net $\left(\int f_\eta dP - \int f_\eta dP_\eta \right)_{\eta \in D}$, we can find $D^\prime$ and $h$ with the desired properties.
\end{proof}

The following lemma is well-known, but we include a proof for the sake of completeness.

\begin{lem}\label{lem:supcontractive}
Let $(f_i)_{i \in I}$ be a collection of contractions of $S$ into $\mathbb{R}$ such that for each $x \in S$ the set $\left\{f_{i}(x) \mid i \in I\right\}$ is bounded. Then the maps $\sup_{i \in I} f_i$ and $\myinf_{i \in I} f_i$ are also contractions. 
\end{lem}

\begin{proof}
Fix $x, y \in S$. For each $i \in I$, $f_i$ being a contraction,
$$f_i(x) \leq f_i(y) + d(x,y),$$
whence
$$\sup_{i \in I} f_i(x) \leq \sup_{i \in I} f_i(y) + d(x,y),$$
and, reversing the role of $x$ and $y$, also
$$\sup_{i \in I} f_i(y) \leq \sup_{i \in I} f_i(x) + d(x,y),$$
from which we conclude dat $\sup_{i \in I} f_i$ is a contraction. 

Analogously, one proves that $\myinf_{i \in I} f_i$ is a contraction.

\end{proof}

\begin{lem}\label{lem:DiniApp}
Let $K \subset S$ be compact, $(g_\eta)_{\eta}$ a net of contractions of $K$ into $S$, and $g$ a contraction of $K$ into $S$.  Suppose that 
$$\mylim_{\eta} \sup_{x \in K} \left|g(x) - g_\eta(x)\right| = 0.$$
Then the net $(g_\eta)_{\eta}$ is uniformly bounded on $K$, and, for each $\eta$, the maps $\sup_{\zeta \succeq \eta} g_\zeta$ and $\myinf_{\zeta \succeq \eta} g_\zeta$ are contractions. Furthermore,
\begin{equation*}
\mylim_{\eta} \sup_{x \in K} \left|\left(\sup_{\zeta \succeq \eta} g_\zeta\right)(x) - \left(\myinf_{\zeta \succeq \eta} g_\zeta \right)(x)\right| = 0.
\end{equation*}
\end{lem}

\begin{proof}
By Lemma \ref{lem:supcontractive}, for each $ \eta$, the maps $\sup_{\zeta \succeq \eta} g_\zeta$ and $\myinf_{\zeta \succeq \eta} g_\zeta$ are contractions. Furthermore, for each $x \in K$, $g(x) = \lim_{\eta}g_\eta(x),$ whence $$g(x) = \limsup_\eta g_\eta(x) = \liminf_\eta g_\eta(x). $$
In particular, the net $\left(\sup_{\zeta \succeq \eta} g_\zeta\right)_\eta$ is monotincally decreasing to $g$, and  the net $\left(\myinf_{\zeta \succeq \eta} g_\zeta\right)_\eta$ is monotonically increasing to $g$. Therefore, the net $\left(\sup_{\zeta \succeq \eta} g_\zeta - \myinf_{\zeta \succeq \eta} g_\zeta\right)_\eta$ is monotonically decreasing to $0$. Now the proof is finished by Dini's Theorem (\cite{K75},\cite{KR01}).
\end{proof}

We now provide the proof of Theorem \ref{thm:AkW}.

\begin{proof}[Proof of Theorem \ref{thm:AkW}]
Fix a net $(P_\eta)_{\eta \in D}$ in $\mathcal{P}^1(S)$ and $P \in \mathcal{P}^1(S)$.

One readily sees that 

$$\lambda_\kappa(P_\eta \rightarrow P) \leq \lambda_W(P_\eta \rightarrow P).$$

We now establish the reverse inequality, which will finish the proof. Fix $\varepsilon > 0$. According to Lemma \ref{lem:seqsup}, choose a net  $(f_\eta)_{\eta \in D}$ in $\kappa_a(S,\mathbb{R})$, a directed set $D'$, and a monotonically increasing and cofinal map $h :  D' \rightarrow D$ such that
\begin{equation}
\lambda_{W}(P_\eta \rightarrow P) -  \varepsilon  \leq \lim_{\eta'} \left(\int f_{h(\eta')} dP - \int f_{h(\eta')} dP_{h(\eta')}\right) \\ \leq \lambda_{W}(P_\eta \rightarrow P).\label{eq:LWStarSeq}
\end{equation}
 Furthermore, finite Borel measures on a separable and completely metrizable topological space being tight (\cite{P67}), take a compact set $K \subset S$ such that
\begin{equation}
\int_{S \setminus K} d(a,\cdot) dP < \varepsilon,\label{eq:outsideK}
\end{equation}
and assume without loss of generality that $a \in K$. By Ascoli's Theorem (\cite{W70}), the collection $\kappa_a(K,\mathbb{R})$ is uniformly compact, which allows us to fix $f \in \kappa_a(K,\mathbb{R})$, a directed set $D''$, and a monotonically increasing and cofinal map $h' : D'' \rightarrow D'$ such that, putting 
\begin{equation}
k = h \circ h^\prime,\label{eq:defk}
\end{equation} 
\begin{equation}
\mylim_{\eta^{\prime \prime}} \sup_{x \in K} \left|f(x) - f_{k(\eta'')}(x)\right| = 0.\label{eq:AscApp}
\end{equation}
Notice that, for all $\eta^{\prime \prime} \in D^{\prime \prime}$, since $f_{k(\eta^{\prime \prime})} \in \kappa_a(S,\mathbb{R})$, $\left|f_{k(\eta^{\prime \prime})}\right| \leq d(a,\cdot)$, whence, by Lemma \ref{lem:supcontractive}, $\sup_{\zeta^{\prime \prime} \succeq {\eta''}} f_{k(\zeta^{\prime \prime})}$ and $\myinf_{\zeta'' \succeq \eta''} f_{k(\zeta'')}$ also belong to $\kappa_a(S,\mathbb{R})$. Furthermore, by (\ref{eq:AscApp}) and Lemma \ref{lem:DiniApp},
\begin{equation*}
\mylim_{\eta^{\prime \prime}} \sup_{x \in K} \left|\left(\sup_{\zeta'' \succeq \eta''} f_{k(\zeta'')}\right)(x) - \left(\myinf_{\zeta'' \succeq \eta''} f_{k(\zeta'')}\right)(x)\right| = 0,
\end{equation*}
so that we can fix $\eta^{\prime \prime}_0 \in D^{\prime \prime}$ such that
\begin{equation}
 \sup_{x \in K} \left|\left(\sup_{\zeta'' \succeq \eta''_0} f_{k(\zeta'')}\right)(x) - \left(\myinf_{\zeta'' \succeq \eta''_0} f_{k(\zeta'')}\right)(x)\right|  < \varepsilon.\label{eq:unifcvcon}
\end{equation} 
For each $\eta^{\prime \prime} \in D^{\prime \prime}$ with $\eta^{\prime \prime} \succeq \eta^{\prime \prime}_0$, 
\begin{eqnarray}
\lefteqn{\left(\int f_{k(\eta'')} dP - \int f_{k(\eta'')} dP_{k(\eta'')}\right)}\label{eq:startbigineq}\\
&\leq& \int \sup_{\zeta'' \succeq \eta''_0} f_{k(\zeta'')} dP - \int \myinf_{\zeta'' \succeq \eta''_0} f_{k(\zeta'')} dP_{k(\eta'')}\nonumber\\
&=& \int \left(\sup_{\zeta'' \succeq \eta''_0} f_{k(\zeta'')} - \myinf_{\zeta'' \succeq \eta''_0} f_{k(\zeta'')}\right) dP\nonumber\\
&& + \left( \int \myinf_{\zeta'' \succeq \eta''_0} f_{k(\zeta'')} dP - \int \myinf_{\zeta''\succeq \eta''_0} f_{k(\zeta'')} dP_{k(\eta'')}\right)\nonumber\\
&=& \int_{K} \left(\sup_{\zeta'' \succeq \eta''_0} f_{k(\zeta'')} - \myinf_{\zeta'' \succeq \eta''_0} f_{k(\zeta'')}\right) dP\nonumber\\
&& + \int_{S \setminus K} \left(\sup_{\zeta'' \succeq \eta''_0} f_{k(\zeta'')} - \myinf_{\zeta'' \succeq \eta''_0} f_{k(\zeta'')}\right) dP\nonumber\\
&& + \left( \int \myinf_{\zeta'' \succeq \eta''_0} f_{k(\zeta'')} dP - \int \myinf_{\zeta'' \succeq \eta''_0} f_{k(\zeta'')} dP_{k(\eta'')}\right).\nonumber
\end{eqnarray}
By (\ref{eq:unifcvcon}),
\begin{equation}
\int_{K} \left(\sup_{\zeta'' \succeq \eta''_0} f_{k(\zeta'')} - \myinf_{\zeta'' \succeq \eta''_0} f_{k(\zeta'')}\right) dP < \varepsilon.\label{eq:onK}
\end{equation}
Since $\sup_{\zeta'' \succeq \eta''_0} f_{k(\zeta'')}$ and $\myinf_{\zeta'' \succeq \eta''_0} f_{k(\zeta'')}$ belong to $\kappa_a(S,\mathbb{R})$, and using (\ref{eq:outsideK}),
\begin{equation}
\int_{S \setminus K} \left(\sup_{\zeta'' \succeq \eta''_0} f_{k(\zeta'')} - \myinf_{\zeta'' \succeq \eta''_0} f_{k(\zeta'')}\right) dP \leq 2 \int_{S \setminus K} d(a,\cdot) dP < 2 \varepsilon.\label{eq:AppOutsideK}
\end{equation}
Plugging in (\ref{eq:onK}) and (\ref{eq:AppOutsideK}) in (\ref{eq:startbigineq}), taking superior limits, and again using the fact that $\myinf_{\zeta'' \succeq \eta''_0} f_{k(\zeta'')}$ belongs to $\kappa_a(S,\mathbb{R})$,
\begin{eqnarray}
\lefteqn{\limsup_{\eta^{\prime \prime}} \left(\int f_{k(\eta'')} dP - \int f_{k(\eta'')} dP_{k(\eta'')}\right)} \label{eq:AlmostDone}\\
&\leq&3 \varepsilon + \limsup_{\eta^{\prime \prime}} \left( \int \myinf_{\zeta'' \succeq \eta''_0} f_{k(\zeta'')} dP - \int \myinf_{\zeta'' \succeq \eta''_0} f_{k(\zeta'')} dP_{k(\eta'')}\right)\nonumber\\
&\leq& 3 \varepsilon + \lambda_\kappa(P_{k(\eta'')} \rightarrow P)\nonumber\\
&\leq& 3 \varepsilon + \lambda_\kappa(P_\eta \rightarrow P).\nonumber
\end{eqnarray}
Using  (\ref{eq:defk}) and the fact that a subnet of a convergent net converges to the same limit point,
\begin{eqnarray*}
\lefteqn{\limsup_{\eta^{\prime \prime}} \left(\int f_{k(\eta'')} dP - \int f_{k(\eta'')} dP_{k(\eta'')}\right)}\\
&& = \lim_{\eta' } \left(\int f_{h(\eta')} dP - \int f_{h(\eta')} dP_{h(\eta')}\right),
\end{eqnarray*}
which, by (\ref{eq:LWStarSeq}) and (\ref{eq:AlmostDone}), yields
\begin{equation*}
\lambda_{W}(P_\eta \rightarrow P) \leq  \lambda_{\kappa}(P_{\eta} \rightarrow P) + 4\varepsilon.
\end{equation*}
This finishes the proof by arbitrariness of $\varepsilon > 0$.
\end{proof}

\section{The relative Hausdorff measure of non-compactness for the Wasserstein metric}\label{sec:HMWM}

In a complete metric space $(X,m)$, the {\em relative Hausdorff measure of non-compactness} of a set $A \subset X$  (\cite{BG80},\cite{WW96}) is given by 
\begin{equation*}
\mu_{\text{\upshape{H}},m}(A) = \myinf_{\substack{Y \subset X\\\textrm{finite}}} \sup_{a \in A} \myinf_{y \in Y} m(y,a),
\end{equation*}
and this measure coincides with the relative measure of non-compactness for the approach structure underlying $m$ (Subsection \ref{subsec:compactness}). One readily verifies that $A$ is $m$-bounded if and only if $\mu_{\textrm{\upshape{H}},m}(A)  < \infty$. Furthermore, by Theorem \ref{thm:compactness}, $A \subset X$ is $m$-relatively compact if and only if $\mu_{\textrm{\upshape{H}},m}(A) = 0$.

The relative (Hausdorff) measure of non-compactness of a set of probability measures was studied for the weak approach structure in \cite{BLV11}, for the continuity approach structure in \cite{B16}, and for the parametrized Prokhorov metric in \cite{B16I}.

Here we are interested in finding a meaningful expression for the relative Hausdorff measure of non-compactness for the Wasserstein metric. More precisely, keeping the notation from the previous section, we will study, for a set $\Gamma \subset \mathcal{P}^1(S)$, the number 
\begin{equation}
\mu_{\text{\upshape H},W}(\Gamma) = \myinf_{\substack{\Phi \subset \mathcal{P}^1(S)\\\textrm{finite}}} \sup_{P \in \Gamma} \myinf_{Q \in \Phi} W(P,Q).\label{eq:HMWdef}
\end{equation}
In doing so, we will make use of the following corollary of Theorem \ref{thm:AkW}.

\begin{thm}
For $\Gamma \subset \mathcal{P}^1(S)$,
\begin{equation}
\mu_{\text{\upshape H},W}(\Gamma) = \sup_{(P_\eta)_\eta} \myinf_{(P_{h(\eta^\prime)})_{\eta^\prime}} \myinf_{P \in \mathcal{P}^1(S)} \sup_{f \in \kappa(S,\mathbb{R})} \limsup_{\eta^\prime} \left|\int f dP - \int f dP_{h(\eta')}\right|,\label{eq:ExprHMW}
\end{equation}
the supremum taken over all nets $(P_\eta)_\eta$ in $\Gamma$, and the first infimum over all subnets $(P_{h(\eta^\prime)})_{\eta^\prime}$ of $(P_\eta)_\eta$. 
\end{thm}

\begin{proof}
By Theorem \ref{thm:AkW}, the contractive approach structure $\mathcal{A}_\kappa$ is metrized by the Wasserstein metric $W$. Therefore, $\mu_{\text{\upshape{H}},W}$ coincides with the relative measure of non-compactness for $\mathcal{A}_\kappa$, which, by Theorem \ref{thm:compactness} and Proposition \ref{pro:limopkappa}, coincides with the right-hand side of (\ref{eq:ExprHMW}).
\end{proof}

We say that a collection $\Gamma \subset \mathcal{P}^1(S)$ is {\em uniformly integrable} iff there exists $a \in S$ such that for each $\varepsilon > 0$ there exists a bounded set $B \subset S$ such that for each $P \in \Gamma$
$$\int_{S \setminus B} d(a,\cdot) dP < \varepsilon,$$
and we define the {\em measure of non-uniform integrability} by 
$$\mu_{\text{\upshape{UI}}}(\Gamma) = \myinf_{a \in S} \myinf_{B} \sup_{P \in \Gamma} \int_{S \setminus B} d(a,\cdot) dP,$$
the second infimum taken over all bounded sets $B \subset S$. Of course, $\mu_{\text{\upshape{UI}}}(\Gamma) = 0$ if and only if $\Gamma$ is uniformly integrable.

For $a \in S$ and $R \in \mathbb{R}^+_0$, let $B(a,R)$ stand for the {\em open ball with center $a$ and radius $R$}, that is 
$$B(a,R) = \{x \in S \mid d(a,x) < R\},$$
and $B^\star(a,R)$ for the {\em closed ball with center $a$ and radius $R$}, that is
$$B^\star(a,R) = \{x \in S \mid d(a,x) \leq R\}.$$

\begin{lem}\label{lem:Sioen}
Fix $a \in S$ and $R \in \mathbb{R}^+_0$. Define the map $\phi_{a,R} : S \rightarrow \mathbb{R}$ by
\begin{displaymath}
\phi_{a,R}(x) = \left\{\begin{array}{clrr}      
\left(1 - \frac{R}{d(a,x)}\right)^+ &\textrm{ if }& x \neq a\\       
0 &\textrm{ if }& x = a
\end{array}\right..
\end{displaymath}
Then, for each $0 < \varepsilon < 1$,
\begin{equation*}
(1 - \varepsilon) 1_{S \setminus B^\star(a,R/\varepsilon)} \leq \phi_{a,R} \leq 1_{S \setminus B^\star(a,R)},
\end{equation*}
and the map  $\psi_{a,R} : S \rightarrow \mathbb{R}$, defined by 
$$\psi_{a,R}(x) = \phi_{a,R}(x) d(a,x),$$ 
is a contraction.
\end{lem}

\begin{proof}
This follows by straightforward verification.
\end{proof}

\begin{lem}\label{lem:fRgR}
Let $a \in \mathbb{R}$ and $f \in \kappa_a(S,\mathbb{R})$. Put, for each $R \in \mathbb{R}^+_0$,
$$f_R = (f \wedge R) \vee (-R) \text{ and } g_R = f - f_R.$$
Then $f_R \rightarrow f$ pointwise as $R \rightarrow \infty$, and, for each $R \in \mathbb{R}^+_0$, the following assertions hold.
\begin{itemize}
\item[(a)] $f_R \in \mathcal{C}_b(S,\mathbb{R}),$
\item[(b)] $g_R \in \kappa_a(S,\mathbb{R}),$
\item[(c)] $f = f_R + g_R,$
\item[(d)] $\left|f_R\right| \leq \left|f\right|,$
\item[(e)] $g_R = g_R 1_{S \setminus B^\star(a,R)}.$
\end{itemize}
\end{lem}

\begin{proof}
(a), (c), and (d) are trivial.

To prove (b), notice that 
\begin{displaymath}
f_R(x) = \left\{\begin{array}{clrrr}      
- R &\textrm{ if }& f(x) < -R\\       
f(x) &\textrm{ if }& -R \leq f(x) \leq R\\
R &\textrm{ if }& R < f(x)
\end{array}\right.
\end{displaymath}
and
\begin{equation}
g_R(x) = \left\{\begin{array}{clrrr}      
f(x) + R &\textrm{ if }& f(x) < -R\\       
0 &\textrm{ if }& -R \leq f(x) \leq R\\
f(x) - R &\textrm{ if }& R < f(x)
\end{array}\right..\label{eq:gR}
\end{equation}
Take $x, y \in S$. 

If $f(x) <  - R$ and $- R \leq f(y) \leq R$, then, $f$ being a contraction,
$$-d(x,y) \leq  f(x) - f(y) \leq f(x) + R = g_R(x) - g_R(y) = f(x) + R < 0,$$
from which we conclude that
$$\left|g_R(x) - g_R(y)\right| \leq d(x,y).$$

If $f(x) < - R$ and $f(y) > R$, then, $f$ being a contraction,
\begin{eqnarray*}
0 < f(y) - f(x) - 2R  &=& g_R(y) - g_R(x)\\
 &=& f(y) - f(x) - 2R < f(y) - f(x) \leq d(x,y),
\end{eqnarray*}
from which we again conclude that
$$\left|g_R(x) - g_R(y)\right| \leq d(x,y).$$

The other cases are dealt with analogously, which finishes the proof of (b).

To establish (e), observe that, since $f \in \kappa_a(S,\mathbb{R})$, $\left|f\right| \leq d(a,\cdot)$. Therefore, if $x \in B^\star(a,R)$, then $-R \leq f(x) \leq R$, whence, by (\ref{eq:gR}), $g_R(x) = 0$. This proves (e).

\end{proof}

We will now prove the main result of this section. Recall that a set $\Gamma \subset \mathcal{P}(S)$ is said to be {\em tight} iff for each $\varepsilon > 0$ there exists a compact set $K \subset S$ such that $P(S \setminus K) < \epsilon$ for each $P \in \Gamma$.

\begin{thm}\label{thm:muUImuH}
For $\Gamma \subset \mathcal{P}^1(S)$,
\begin{equation}
\mu_{\text{\upshape{UI}}}(\Gamma) \leq \mu_{\text{\upshape{H}},W}(\Gamma).\label{eq:lbW}
\end{equation}
Moreover, the inequality in (\ref{eq:lbW}) becomes an equality if $\Gamma$ is tight.
\end{thm}

\begin{proof}

Suppose that, for $\alpha > 0$,
\begin{equation}
\mu_{\text{\upshape{H}},W}(\Gamma) < \alpha,\label{eq:hwalpha}
\end{equation}
and fix $a \in S$ and $0 < \varepsilon < 1$. By (\ref{eq:defWass}) and (\ref{eq:HMWdef}), there exists a finite set $\Phi \subset \mathcal{P}^1(S)$ such that for each $P \in \Gamma$ there exists $Q \in \Phi$ with the property that 
\begin{equation}
\sup_{f \in \kappa(S,\mathbb{R})} \left|\int f dP - \int f dQ\right| < \alpha.\label{eq:GammaFinite}
\end{equation}
Since $\Phi$ is finite, we can choose $R > 0$ such that
\begin{equation}
\forall Q \in \Phi : \int_{S \setminus B^\star(a,R)} d(a,\cdot) dQ < \varepsilon.\label{eq:leqepsilon}
\end{equation}
Furthermore, according to Lemma \ref{lem:Sioen}, fix $\phi_{a,R} : S \rightarrow \mathbb{R}$ such that
\begin{equation}
(1 - \varepsilon) 1_{S \setminus B^\star(a,R/\varepsilon)} \leq \phi_{a,R} \leq 1_{S \setminus B^\star(a,R)}\label{eq:phibetween}
\end{equation}
and 
\begin{equation}
\psi_{a,R} = \phi_{a,R} d(a,\cdot) \in \kappa(S,\mathbb{R}).\label{eq:contractivemap}
\end{equation}
Fix $P \in \Gamma$. Take $Q \in \Phi$ such that (\ref{eq:GammaFinite}) holds. Then, by (\ref{eq:phibetween}), (\ref{eq:contractivemap}), and (\ref{eq:leqepsilon}), 
\begin{eqnarray*}
\lefteqn{(1 - \varepsilon) \int_{S \setminus B^\star(a,R/\varepsilon)} d(a,\cdot) dP}\\
&\leq& \int_S \phi_{a,R} d(a,\cdot) dP\\
&<& \int_S \phi_{a,R} d(a,\cdot) dQ + \alpha\\ 
&\leq& \int_{S \setminus B^\star(a,R)} d(a,\cdot) dQ + \alpha\\
&<& \varepsilon + \alpha,
\end{eqnarray*}
which proves that 
$$\mu_{\text{\upshape{UI}}}(\Gamma) < (\varepsilon + \alpha)/(1 - \varepsilon).$$
Since $\alpha$ was arbitrarily chosen such that (\ref{eq:hwalpha}) holds, we have shown (\ref{eq:lbW}).

Now assume that $\Gamma$ is tight. We will prove that
\begin{equation}
\mu_{\text{\upshape{H}},W}(\Gamma) \leq \mu_{\text{\upshape{UI}}}(\Gamma).\label{eq:HWUnderUI}
\end{equation}
Suppose that, for $\alpha > 0$, 
\begin{equation}
\mu_{\text{\upshape{UI}}}(\Gamma) < \alpha,\label{eq:UIleqAlpha}
\end{equation}
that is, there exist $a \in S$ and a bounded set $B \subset S$ such that
\begin{equation}
\forall P \in \Gamma : \int_{S \setminus B} d(a,\cdot) dP < \alpha.\label{eq:onBleqalpha}
\end{equation}
Take an arbitrary net $(P_\eta)_\eta$ in $\Gamma$. Since $\Gamma$ is tight, it is, by Prokhorov's Theorem (\cite{P67},\cite{B99}), weakly relatively compact. We thus find a subnet $(P_{h(\eta')})_{\eta'}$ and $P \in \mathcal{P}(S)$ such that $P_{h(\eta')} \stackrel{w}{\rightarrow} P$. 

We first show that $P \in \mathcal{P}^1(S)$. For $R \in \mathbb{R}^+$, 
\begin{equation}
\int_S \left(d(a,\cdot) \wedge R\right) dP = \lim_{\eta^\prime} \int_S \left(d(a,\cdot) \wedge R\right) dP_{h(\eta^\prime)} \leq \liminf_{\eta^\prime} \int_S d(a,\cdot) dP_{h(\eta^\prime)}.\label{eq:FinMom1}
\end{equation}
Furthermore, by the fact that $B$ is bounded and (\ref{eq:onBleqalpha}), for each $\eta^\prime$,
\begin{equation}
\int_S d(a,\cdot) dP_{h(\eta^\prime)} = \int_{B} d(a,\cdot) dP_{h(\eta^\prime)} + \int_{S \setminus B} d(a,\cdot) dP_{h(\eta^\prime)} \leq \sup_{x \in B} d(a,x)  + \alpha.\label{eq:FinMom2}
\end{equation}
Letting $R \uparrow \infty$ and using the Monotone Convergence Theorem, (\ref{eq:FinMom1}) and (\ref{eq:FinMom2}) yield
$$\int_S d(a,\cdot) dP \leq \sup_{x \in B} d(a,x)  + \alpha,$$
and we conclude that $P \in \mathcal{P}^1(S)$.

We now establish that
\begin{equation}
\lambda_\kappa(P_{h(\eta')} \rightarrow P) < \alpha.\label{eq:LambdaKappaUnderAlpha}
\end{equation}
Fix $f \in \kappa_a(S,\mathbb{R})$ and put, for each $R \in \mathbb{R}^+$, 
$$f_R = (f \wedge R) \vee (-R) \text{ and } g_R = f - f_R.$$
Then $f_R \rightarrow f$ pointwise as $R \rightarrow \infty$, and the assertions (a)--(e) in Lemma \ref{lem:fRgR} hold. By (b), (c), and (e) in Lemma \ref{lem:fRgR}, for all $R \in \mathbb{R}^+_0$ and $\eta'$,
\begin{eqnarray*}
\int f dP_{h(\eta')} &=& \int f_R dP_{h(\eta')} + \int g_R dP_{h(\eta')}\\
 &\leq& \int f_R dP_{h(\eta')} + \int_{S \setminus B^\star(a,R)} d(a,\cdot) dP_{h(\eta')},
\end{eqnarray*}
which, by the fact that $P_{h(\eta^\prime)} \stackrel{w}{\rightarrow} P$ and (a) in Lemma \ref{lem:fRgR}, yields
\begin{equation*}
\limsup_{\eta^\prime}  \int f dP_{h(\eta')} \leq \int f_R dP + \limsup_{\eta^\prime} \int_{S \setminus B^\star(a,R)} d(a,\cdot) dP_{h(\eta')}.
\end{equation*}
Letting $R \uparrow \infty$ and using (d) in Lemma \ref{lem:fRgR}, the Dominated Convergence Theorem, and (\ref{eq:onBleqalpha}), we infer
\begin{equation*}
\limsup_{\eta^\prime} \int f dP_{h(\eta')} \leq \int f dP + \alpha,
\end{equation*}
which, by (\ref{eq:NewFormLambdaKappa}), gives (\ref{eq:LambdaKappaUnderAlpha}). Now it follows from (\ref{eq:ExprHMW}) that 
$$\mu_{\text{\upshape{H}},W}(\Gamma) \leq \alpha.$$ 
Since $\alpha$ was arbitrarily chosen such that (\ref{eq:UIleqAlpha}) holds, we have shown (\ref{eq:HWUnderUI}).

\end{proof}

We will give an example which shows that the tightness condition in Theorem \ref{thm:muUImuH} cannot be omitted. For $a \in S$, let $\delta_a$ be the Dirac measure on $S$ putting all its mass on $a$

\begin{lem}\label{lem:WonDeltax}
For $a,b \in S$, 
$$W(\delta_a,\delta_b) = d(a,b).$$
\end{lem}

\begin{proof}
We have
$$W(\delta_a,\delta_b) = \sup_{f \in \kappa(S,\mathbb{R})}\left|f(a) - f(b)\right| \leq d(a,b).$$
Moreover, for $f = d(a,\cdot)$,
$$\left|f(a) - f(b)\right| = d(a,b),$$
which proves the desired equality.
\end{proof}

\begin{lem}\label{lem:ConWPDeltaa}
Let $P \in \mathcal{P}^1(S)$, $a\in S$, $M \in \mathbb{R}^+_0$, and $0 < \varepsilon \leq M$. Then
$$P(B(a,M)) < \epsilon/M \Rightarrow W(P,\delta_a) > M - \epsilon.$$
\end{lem}

\begin{proof}
Suppose that 
$$W(P,\delta_a) \leq M - \epsilon.$$ 
Then
\begin{eqnarray*}
M P(S \setminus B(a,M)) &\leq& \int_{S \setminus B(a,M)} d(a,\cdot) dP\\
&\leq& \int d(a,\cdot) dP\\
&=& \left|\int d(a,\cdot) dP - \int d(a,\cdot) d\delta_a\right|\\
&\leq& M - \epsilon,
\end{eqnarray*}
whence 
$$P(B(a,M)) \geq \epsilon/M,$$
which finishes the proof.
\end{proof}

Let $\mathcal{C}$ be the space of continuous maps of the compact interval $[0,1]$ into $\mathbb{R}$, equipped with the supremum norm $\|\cdot\|_\infty$, defined by
$$\|x\|_\infty = \sup_{t \in [0,1]} \left|x(t)\right|.$$
Then $\mathcal{C}$ is a separable Banach space. 

\begin{thm}
For each $M \in \mathbb{R}^+_0$ there exists $\Gamma \subset \mathcal{P}^1(\mathcal{C})$ such that 
$$\mu_{\text{\upshape UI}}(\Gamma) = 0 \text{ and } \mu_{\text{\upshape{H}},W}(\Gamma) = M.$$
\end{thm}

\begin{proof}
Define, for $n \in \mathbb{N}_0$,
\begin{displaymath}
x_n(t) = \left\{\begin{array}{clrrr}      
M &\textrm{ if }& 0 \leq t \leq 1 - \frac{1}{n}\\       
- 2 M n (n+1) t + 2 M n^2 - M &\textrm{ if }& 1 - \frac{1}{n}\leq t \leq 1 - \frac{1}{n+1}\\
-M &\textrm{ if }& 1 - \frac{1}{n+1} \leq t \leq 1
\end{array}\right..
\end{displaymath}
Then, for all $m, n \in \mathbb{N}_0$,
\begin{equation}
\|x_n \|_\infty = M\label{eq:nxnleqM}
\end{equation}
and
\begin{equation}
m \neq n \Rightarrow \|x_m - x_n\|_\infty = 2 M.\label{eq:pwdisjoint}
\end{equation}
Put 
$$\Gamma = \{\delta_{x_n} \mid n \in \mathbb{N}_0\}.$$
By (\ref{eq:nxnleqM}),
$$\mu_{\text{\upshape{UI}}}(\Gamma) = 0,$$
and, by Lemma \ref{lem:WonDeltax}, for each $n \in \mathbb{N}_0$,
$$W(\delta_{x_n},\delta_0) =  \|x_n\|_\infty = M,$$
whence, by (\ref{eq:HMWdef}),
\begin{equation}
\mu_{\text{\upshape{H}},W}(\Gamma) \leq M.\label{eq:muWleqM}
\end{equation}
Now, for $0 < \epsilon \leq M$, fix a finite collection $\Phi \subset \mathcal{P}^1(S)$. By (\ref{eq:pwdisjoint}),
$$n \neq m \Rightarrow B_{\mathcal{C}}(x_n,M) \cap B_{\mathcal{C}}(x_m,M) = \emptyset.$$
Thus
$$\cup_{k \geq n} B_{\mathcal{C}}(x_k,M) \downarrow \emptyset \text{ as } n \rightarrow \infty,$$
whence there exists $n_0$ such that, for all $Q \in \Phi$,
$$Q(B_{\mathcal{C}}(x_{n_0}, M)) < \epsilon/M,$$
and by Lemma \ref{lem:ConWPDeltaa} we infer that
$$W(Q,\delta_{x_{n_0}}) > M - \epsilon.$$
From (\ref{eq:HMWdef}) we now deduce that
$$\mu_{\text{\upshape H},W}(\Gamma) \geq M - \epsilon,$$
which, combined with (\ref{eq:muWleqM}), gives
$$\mu_{\text{\upshape{H}},W}(\Gamma) = M.$$
This finishes the proof.
\end{proof}

\end{document}